\documentclass[12pt,centertags,psamsfonts,reqno]{amsart}    
\usepackage{amssymb}
\usepackage[mathcal]{euscript}

\newtheorem{thm}{Theorem}
\newtheorem{lemma}{Lemma}
\newtheorem{cor}{Corollary}

\theoremstyle{remark}
\newtheorem{rem}{Remark}

\newcommand{\mr}{{\mathbb R}}

\newcommand{\mn}{{\mathbb N}}

\newcommand{\mc}{{\mathbb C}}
\newcommand{\md}{{\mathbb D}}

\newcommand{\mt}{{\mathbb T}}

\renewcommand{\rho}{\varrho}

\newcommand{\hil}{\mathcal{H}}

\newcommand{\num}{\mathop{\rm Num}\nolimits}
\newcommand{\dist}{\mathop{\rm dist}\nolimits}
\newcommand{\im}{\mathop{\rm Im}}

\begin{document}

\title[From spectral theory to zeros of holomorphic functions]{From spectral theory to bounds on zeros of holomorphic functions}

\author[M. Hansmann]{Marcel Hansmann}
\address{Faculty of Mathematics\\
Chemnitz University of Technology\\
Chemnitz\\
Germany.}
\email{marcel.hansmann@mathematik.tu-chemnitz.de}

\author[G. Katriel]{Guy Katriel}
\address{Department of Mathematics\\
ORT Braude College\\  
Karmiel\\
 Israel}
\email{katriel@braude.ac.il}
  

\begin{abstract}
We show how  eigenvalue estimates for linear operators can be used to obtain new Blaschke
type bounds on zeros of holomorphic functions on the unit disk.
\end{abstract}

\maketitle

\section{Introduction}

The aim of this work is to show that the spectral theory of linear operators can be used to derive results about the distribution of zeros
of a large class of holomorphic functions on the unit disk, namely, the class of all Cauchy transforms of complex Borel measures on the unit circle.
This class, henceforth denoted by  $\mathcal{K}$, consists of all holomorphic functions $h \in H(\md)$ which can be written as
\begin{equation}\label{cauchy}
 h(w)= \int_\mt \frac{\mu(d\zeta)}{1-w\overline{\zeta}},
\end{equation}
where $\mu$ is some finite complex Borel measure on the unit circle $\mt$.  Obviously $\mathcal{K}$  is a linear space, and a natural norm on it is given
by setting $\|h\|_{\mathcal{K}}$ to be the minimal total variation of all measures $\mu$ generating $h$ according to (\ref{cauchy}) (see Section \ref{sec:Cauchy}).

Denoting
the zero set of $h$ by $\mathcal{Z}(h)$, we will prove, assuming $h\in \mathcal{K}$ and $h(0)=1$, that
\begin{equation}\label{sc0}
\sum_{z \in \mathcal{Z}(h)} \left( |z|^{-1}-1 \right)\leq \left\| \frac{h(w)-1}{w}\right\|_{\mathcal{K}},
\end{equation}
where each zero in the sum is counted according to its order. Here it should be noted that when $h\in {\mathcal{K}}$, then
also $\frac{h(w)-h(0)}{w}\in {\mathcal{K}}$ and  that $\|\frac{h(w)-h(0)}{w} \|_{\mathcal{K}}\leq \|h \|_{\mathcal{K}}$. % (see Section \ref{sec:Cauchy}).

The finiteness of the sum in (\ref{sc0}) is certainly well-known, being a simple consequence of the fact that for $h \in \mathcal{K}$ with $h(0)=1$ the corresponding Blaschke-sum $\sum_{z \in \mathcal{Z}(h)} (1-|z|)$ is finite, see \cite{b_cima06}. However, the explicit bound in (\ref{sc0}) is new (to the best of
our knowledge). Moreover, we find it of interest that our result is proved
via purely operator-theoretic arguments, without appeal to the theory of complex functions. The proof proceeds by associating to a given holomorphic function $h$ an explicitly defined linear operator $L$ which is a rank-1 perturbation of a contraction $A$, and which has the property that the eigenvalues of $L$ outside the unit disk are the reciprocals of the zeros of $h$. A general result on eigenvalues of trace class perturbations
of linear operators, which is a special case of
\cite{Hansmann10} Theorem 2.1, is then used to obtain information on the eigenvalues of $L$, and hence on the zeros of $h$.

The question naturally arises whether our result can also be proved using more `standard' complex-function theory. While we do not know the
answer to this question, we do note here one {\it{consequence}} of inequality (\ref{sc0}) which can also be derived via a straightforward complex-analysis approach.
This concerns the subclass $H^1\subset \mathcal{K}$ of holomorphic functions with $L^1$ boundary values (see \cite{Rudin} Chapter 17).
Noting that for $h \in H^1$ we have $\|h\|_{\mathcal{K}} \leq \|h\|_{H^1}$ and so%for $h \in H^1$, we have, for $h\in H^1$ with $h(0)=1$,
$$\left\|\frac{h(w)-h(0)}{w}\right\|_{\mathcal{K}}\leq \left\|\frac{h(w)-h(0)}{w}\right\|_{H^1}=\|h(w)-h(0)\|_{H^1},$$
we see that inequality (\ref{sc0}) implies that for $h \in H^1$ with $h(0)=1$%, in the special case that $h\in H^1$ with $h(0)=1$,
 \begin{equation}\label{sc}
   \sum_{z \in \mathcal{Z}(h)} \left( |z|^{-1}-1 \right) \leq \|h-1\|_{H^1}.
 \end{equation}
We show now that (\ref{sc}) can also be obtained, independently of the method developed in this paper, via the classical Jensen formula.
We are grateful to Leonid Golinskii for pointing this out.
The Jensen formula says
that, for any holomorphic function on the unit disk with $h(0)=1$, we have
\begin{equation}\label{jensen}
\prod_{z\in \mathcal{Z}(h)}|z|^{-1}= \sup_{0<r<1} \exp\left( \int_{\mt} \log|h(r\zeta)| \: m(d\zeta)\right),
\end{equation}
where $m(d\zeta)$ denotes normalized Lebesgue measure on $\mt$. Assuming $h\in H^1$ and using convexity, the right hand side of (\ref{jensen}) is bounded from above by $\|h\|_{H^1}$,
while the left-hand side is bounded from below by
$$\prod_{z\in \mathcal{Z}(h)}|z|^{-1}=\prod_{z\in \mathcal{Z}(h)}\left[1+(|z|^{-1}-1)\right]\geq 1+\sum_{z\in \mathcal{Z}(h)}(|z|^{-1}-1),$$
so (\ref{jensen}) implies
$$\sum_{z\in \mathcal{Z}(h)}(|z|^{-1}-1) \leq \|h \|_{H^1}-1\leq \|h-1\|_{H^1},$$
so that we have (\ref{sc}).

It is certainly an interesting question whether our general inequality (\ref{sc0}) can also be derived using Jensen's formula. Whatever the answer turns out to
be, our approach here, which does not invoke any results from complex analysis, gives an alternative pathway to obtaining information on zeros of Cauchy transforms.

The approach described in this paper crucially depends on the fact (see Theorem \ref{thm:4} below) that the class of all $h \in \mathcal{K}$ with $h(0) \neq 0$  consists precisely of those functions which can be represented as
\[
h(w)/h(0)=1+ w\langle (I-wA)^{-1} \phi, \psi \rangle, \quad w \in \md,
\]
where $\phi$ and $\psi$ are arbitrary elements of a complex separable Hilbert space $(\hil, \langle . , . \rangle)$ and  $A$ is a contraction on $\hil$, i.e. a bounded linear operator with $\|A\| \leq 1$.  In the next section we shall prove the following inequality on zeros of such functions.

\begin{thm}\label{thm:1}
Let $(\hil, \langle . ,. \rangle)$ be a complex separable Hilbert space. Moreover, let $\phi, \psi \in \hil$ and let $A$ be a contraction on $\hil$. Define $h \in H(\md)$ as
\[
  h(w)= 1+ w \langle (I-wA)^{-1} \phi, \psi \rangle, \qquad w \in \md.
\]
Then
\[
  \sum_{z \in \mathcal{Z}(h)} \left( |z|^{-1}-1 \right) \leq \|\phi\|_\hil\|\psi\|_\hil,
\]
where each zero is counted according to its order.
\end{thm}
Our main theorem on the zeros of Cauchy transforms, which will follow from Theorem \ref{thm:1} in view of the equivalence noted above, is proved in Section \ref{sec:Cauchy}.

\section{The proof of Theorem \ref{thm:1}}\label{sec:abstract}

In the following $A$ and $L$ denote bounded linear operators on a complex separable Hilbert space $(\hil,\langle .,. \rangle)$
(the inner product being linear in the first component). The numerical range of $A$ is defined as
\[ \num(A) = \{ \langle Af,f\rangle: f \in \hil, \|f\|_\hil=1\}\]
and we clearly have $\num(A) \subset \{ \lambda: |\lambda| \leq \|A\|\}$. The discrete spectrum of $L$, $\sigma_d(L)$, is the set of all isolated eigenvalues of $L$ of finite algebraic multiplicity. Assuming that $L-A$ is a trace class operator we denote its trace norm by $\|L-A\|_{tr}.$

The next result, which is a special case of \cite{Hansmann10} Theorem 2.1, is the main ingredient in the proof of Theorem \ref{thm:1}. For the convenience of the reader, its proof is included in the appendix.
\begin{thm}\label{thm:3}
  Let $A$ and $L$ be as above and assume that $L-A$ is a trace class operator. Then
\[
    \sum_{\lambda \in \sigma_d(L)} \dist(\lambda, \num(A)) \leq \|L-A\|_{tr},
\]
where each eigenvalue is counted according to its algebraic multiplicity.
\end{thm}

We now assume that $\|A\| \leq 1$ and that $M$ is the rank-1 operator on $\hil$ defined by
\[ Mf = - \langle f, \psi \rangle \phi, \quad f \in \hil, \]
where $\phi,\psi \in \hil \setminus \{0\}$. Moreover, we set
$$L=A+M.$$
\begin{rem}
We note that the spectrum of the contraction $A$ is contained in $\overline{\md}$. Moreover, the assumption that $M=L-A$ is a rank-1 operator implies that the essential spectra of $A$ and $L$ coincide (Weyl's theorem) so the essential spectrum of $L$ must be contained in $\overline{\md}$ as well. This shows that the spectrum of $L$ in the complement of $\overline{\md}$ can consist of isolated eigenvalues of finite algebraic multiplicity only, see e.g. \cite{b_Davies} Theorem 4.3.18.
\end{rem}
Using $\|M\|_{tr}=\|M\|=\|\phi\|_\hil \| \psi\|_\hil$ and $\num(A) \subset \overline{\md}$ we obtain from Theorem \ref{thm:3} that
\begin{equation}\label{eq:3}
   \sum_{\lambda \in \sigma_d(L), |\lambda| > 1} (|\lambda|-1) \leq \sum_{\lambda \in \sigma_d(L)} \dist(\lambda, \num(A)) \leq \|\phi\|_\hil \|\psi\|_\hil.
\end{equation}

The following lemma relates the eigenvalues of $L=A+M$ to the zeros of the function
\[
  h(w)=1+w\langle (I-wA)^{-1} \phi , \psi \rangle, \qquad w \in \md.
\]
\begin{lemma}\label{lemma1}
For $w\in \md$ we have $h(w)=0$ if and only if $w^{-1}\in \sigma_d(L)$, and the order of $w$ as a zero of $h$ is equal to the
algebraic multiplicity of $w^{-1}$ as an eigenvalue of $L$.
\end{lemma}
\begin{proof}
For $|w|<1$, set $\lambda=w^{-1}$, so that $|\lambda| > 1$. If $h(w)=1+\langle (\lambda-A)^{-1}\phi, \psi \rangle=0$ then
\begin{eqnarray*}
 L (\lambda-A)^{-1} \phi &=& (A+M)(\lambda-A)^{-1} \phi = A(\lambda-A)^{-1}\phi - \langle (\lambda-A)^{-1}\phi, \psi \rangle \phi \\
&=& \lambda(\lambda-A)^{-1}\phi - (1+\langle (\lambda-A)^{-1}\phi, \psi \rangle) \phi = \lambda (\lambda-A)^{-1} \phi,
\end{eqnarray*}
so $\lambda \in \sigma_d(L)$. Conversely, if $Lf=\lambda f$ for some $f \neq 0$, then $Mf = (\lambda-A)f$ and setting $g= (\lambda-A)f$ we obtain
\begin{equation}
  \label{eq:NEW}
 g= M (\lambda-A)^{-1} g = - \langle (\lambda-A)^{-1}g, \psi \rangle \phi,
\end{equation}
which implies that $g = c \phi$ for some $c \in \mc \setminus \{0\}$. But plugging $g=c \phi$ back into (\ref{eq:NEW}) we obtain that
$$h(w)= 1+ \langle (\lambda-A)^{-1} \phi, \psi  \rangle =0.$$
The fact that the algebraic multiplicity of $\lambda=w^{-1}$ as an eigenvalue of $L$ coincides with the order of $w$ as a zero of $h$ requires some more work. It follows from the fact that $h$ coincides with the perturbation determinant of $L$ by $A$, i.e.
$$h(1/\lambda)=\det(I-M(\lambda-A)^{-1}),$$ and from standard properties of the perturbation determinant, see \cite{b_Gohberg69} p.173-174.
\end{proof}

From (\ref{eq:3}) and Lemma \ref{lemma1} we obtain
\[ \sum_{z \in \mathcal{Z}(h)} \left(|z|^{-1}-1 \right) \leq \|\phi\|_\hil \|\psi\|_\hil.\]
This concludes the proof of Theorem \ref{thm:1}.

% Since for $f \neq 0$
% \[ Lf= \lambda f \quad \Leftrightarrow \quad (I-(\lambda-A)^{-1}M)f=0,\]
% we see that $\lambda \in \sigma_d(L)$ if and only if
% \[\det(I-(\lambda-A)^{-1}M)=\det(I-M(\lambda-A)^{-1})=0.\]

% A computation of the determinant shows that
% \begin{eqnarray*}
%   \det(I-M(\lambda-A)^{-1}) &=& 1-\left\langle M(\lambda-A)^{-1}\frac{\phi}{\|\phi\|_\hil},\frac{\phi}{\|\phi\|_\hil}\right\rangle  \\
%  &=& 1+\langle (\lambda-A)^{-1}\phi,\psi\rangle = h(1/\lambda).
% \end{eqnarray*}

\section{Cauchy transforms}\label{sec:Cauchy}

Let $\mathcal{M}$ denote the class of all finite complex Borel measures on $\mt$ and for $\mu \in \mathcal{M}$ let $K\mu \in H(\md)$ be defined by
\[
(K\mu)(w) = \int_{\mt} \frac{\mu(d\zeta)}{1-w\overline{\zeta}}, \qquad w \in \md.
\]
$K$ is called the \textit{Cauchy transform} of the measure $\mu$.
\begin{rem}
  For a comprehensive discussion of Cauchy transforms we refer to the monograph \cite{b_cima06}, in which much of what follows is discussed in great detail.
\end{rem}

As above, the vector space of all Cauchy transforms is denoted by $\mathcal{K}$, i.e.
\[
   \mathcal{K}= \left\{ h \in H(\md) : h=K\mu \mbox{ for some } \mu \in \mathcal{M} \right\}.
\]
Setting $R_h=\{ \mu \in \mathcal{M} : K\mu=h\}$ the space $\mathcal{K}$ can be equipped with a norm by defining
\[\| h \|_{\mathcal{K}} = \inf \{ \|\mu \| : \mu \in R_h\}, \quad h \in \mathcal{K},\]
where $\|\mu\|$ denotes the total variation of the measure $\mu$. Moreover, for each $h \in \mathcal{K}$ there exists a unique $\mu_0 \in R_h$ such that
\[ \|h\|_{\mathcal{K}}=\|\mu_0\|,\]
see \cite{b_cima06} Proposition 4.1.4.
%We note that
%\[
%  \bigcup_{p\geq 1} H^p(\md) \subset \mathcal{K} \subset \bigcap_{0<p<1} H^p(\md),
%\]
%both inclusions being strict, and
%\begin{equation}\label{hp}
%\|h\|_{\mathcal{K}} \leq \|h\|_{H^p}\;\;{\mbox{ for }} \;h \in H^p(\md),\;\;p \geq 1.
%\end{equation}
For our main result we will also need the \textit{backward shift operator}, $B$, on $H(\md)$,  which is defined as
\[
   (Bh)(w)= \frac{h(w)-h(0)}{w} ,\qquad h \in H(\md), \quad w \in \md.
\]

\begin{lemma}\label{lem:1}
Let $h \in H(\md)$. Then $h \in \mathcal{K}$ if and only if $Bh \in \mathcal{K}$.
\end{lemma}
The proof of Lemma \ref{lem:1} can be found in the appendix. For the proof of the next lemma see \cite{b_cima06} Proposition 11.3.1.
\begin{lemma}\label{lem:2}
The backward shift operator $B: (\mathcal{K}, \|.\|_{\mathcal{K}}) \to (\mathcal{K}, \|.\|_{\mathcal{K}})$ is bounded and
\[\|B\| = \sup_{h \in \mathcal{K}, h \neq 0} \frac{\|Bh\|_{\mathcal{K}}}{\|h\|_{\mathcal{K}}} = 1.\]
\end{lemma}

Our main result on the zeros of Cauchy transforms is the following theorem.
\begin{thm}\label{thm:2}
   Let $h \in \mathcal{K}$ with $h(0)=1$. Then
   \begin{equation}\label{eq:1}
     \sum_{z \in \mathcal{Z}(h)} \left( |z|^{-1}-1\right) \leq \|Bh\|_{\mathcal{K}},
   \end{equation}
 where each zero is counted according to its order.
\end{thm}
\begin{rem}
  The previous inequality is sharp in the following sense: there exist $h \in \mathcal{K}$ with $h(0)=1$ and $\mathcal{Z}(h) \neq \emptyset$ such that the left- and right-hand sides of (\ref{eq:1}) coincide. For instance, this is the case for $h(w)=1+ \frac{w}{1+w}$ since here $Bh=K\delta_{-1}$, where $\delta_{-1}$ denotes the Dirac measure supported on $\{-1\}$, and so both sides of (\ref{eq:1}) are equal to $1$.
\end{rem}
\begin{proof}[Proof of Theorem \ref{thm:2}]
Let $\mu\in R_{Bh}$ be the unique measure with $\|Bh\|_{\mathcal{K}}=\|\mu\|$. In particular, $h(w)=1+w \cdot (K\mu)(w)$. Noting that $d\mu = \nu d|\mu|$ for some measurable function $\nu : \mt \to \mt$, we choose $\hil= L^2(\mt,d|\mu|)$, $\phi(\zeta)=1$, $\psi(\zeta)= \overline{\nu(\zeta)}$ and we define a unitary operator $A$ on $\hil$ by setting $(Af)(\zeta) = \overline{\zeta}f(\zeta)$. Then
\begin{eqnarray*}
  1+ w \langle (I-wA)^{-1} \phi, \psi \rangle  = 1 + \int_{\mt} \frac {w} {1-w\overline{\zeta}}  \mu(d\zeta) = h(w).
\end{eqnarray*}
Since $\|\phi\|_\hil \| \psi\|_\hil = \|\mu\|=\|Bh\|_{\mathcal{K}}$, Theorem \ref{thm:1} implies that
\[   \sum_{z \in \mathcal{Z}(h)} \left( |z|^{-1}-1 \right) \leq \|Bh\|_{\mathcal{K}}. \]
\end{proof}
%Let us collect some corollaries of Theorem \ref{thm:2}.
In view of Lemma \ref{lem:2} we have:

\begin{cor}\label{cor:2}
  Let $h \in \mathcal{K}$ with $h(0)=1$. Then
\[
    \sum_{z \in \mathcal{Z}(h)} \left( |z|^{-1}-1 \right) \leq \|h\|_{\mathcal{K}}.
\]
\end{cor}
\begin{rem}
It can well happen that $\|Bh\|_{\mathcal{K}} < \|h\|_{\mathcal{K}}$. For instance, if $h \equiv 1$ then $\|Bh\|_{\mathcal{K}}=0$, while $\|h\|_{\mathcal{K}}=1$.
\end{rem}

%In view of (\ref{hp}) and the fact that, for $h\in H^p$,
%$$\|Bh\|_{H^p}=\Big\|\frac{h(w)-h(0)}{w}\Big\|_{H^p}=\|h(w)-h(0)\|_{H^p}$$
%we have:
%\begin{cor}\label{cor:1}
%Let $p \geq 1$ and let $h \in H^p(\md)$ with $h(0)=1$. Then
%\[
%   \sum_{w \in \mathcal{Z}(h)} \left( |w|^{-1}-1 \right) \leq \|h-1\|_{H^p}.
%\]
%\end{cor}
%Since $H^2(\md)$ is isomorphic to $l^2(\mn_0)$, with the $H^2$-norm of $h(w)=\sum_{k\geq 0} a_k w^k$ being equal to the $l^2$-norm of $\{a_k\}_{k \in \mn_0}$, we also obtain the following result.
%\begin{cor}
%Let $h(w)=1+\sum_{k\geq 1} a_k w^k \in H^2(\md)$. Then
%\[
%   \sum_{w \in \mathcal{Z}(h)} \left( |w|^{-1}-1 \right) \leq \left( \sum_{k\geq 1} |a_k|^2 \right)^{1/2}.
%\]
%\end{cor}
%
%\begin{rem}
%Estimates on the Blaschke sum of $H^2$-functions have also been studied from a different perspective. Namely, in \cite{Beauzamy92}, \cite{Girardi94} and \cite{Kelly07}  such estimates where derived in terms of the so called \textit{concentration} of the function. For a deeper discussion of this topic see the cited literature.
%\end{rem}

% We will continue our discussion of Cauchy transforms in Section \ref{sec:Cauchy2}. First, however, let us prove Theorem \ref{thm:1}.

% \section{Cauchy transforms - a second look}\label{sec:Cauchy2}
In the proof of Theorem \ref{thm:2} we have first shown that every $h \in \mathcal{K}$ with $h(0)=1$ can be represented as
$h(w)=1+w\langle (I-wA)^{-1} \phi, \psi \rangle$ for special choices of $\phi$ and $\psi$ and some special unitary operator $A$, and then we have applied Theorem \ref{thm:1} to estimate the corresponding Blaschke sum from above by $\|\phi\|_\hil\|\psi\|_\hil$. However, this choice of $\phi,\psi$ and $A$ is certainly not unique and so there exists the possibility that a different choice (e.g. with a non-unitary $A$) might lead to a better bound on the sum in (\ref{eq:1}), or to a larger class of holomorphic functions which might be treatable. The next theorem will show that this is not the case.

\begin{thm}\label{thm:4}
Let $h \in H(\md)$ with $h(0)=1$. Then the following are equivalent:
  \begin{enumerate}
  \item[(i)] $h \in \mathcal{K}.$
  \item[(ii)] There exists a complex separable Hilbert space $(\hil,\langle . , . \rangle_\hil)$, vectors $\phi, \psi \in \hil$ and a contraction $A$ on $\hil$ such that
    \begin{equation}
      \label{eq:4}
    h(w)= 1 + w\langle (I-wA)^{-1} \phi, \psi \rangle_\hil, \qquad w \in \md.
    \end{equation}
\item[(iii)] There exists a complex separable Hilbert space $(\hil', \langle . , . \rangle_{\hil'})$, vectors $\phi', \psi' \in \hil'$ and a unitary operator $U$ on $\hil'$  such that
    \begin{equation}
      \label{eq:5}
    h(w)= 1 + w\langle (I-wU)^{-1} \phi', \psi' \rangle_{\hil'}, \qquad w \in \md.
    \end{equation}
\end{enumerate}
Moreover, if $h \in \mathcal{K}$ then
\[ \|Bh\|_{\mathcal{K}} = \min \|\phi\|_\hil \|\psi\|_\hil= \min \|\phi'\|_{\hil'}\|\psi'\|_{\hil'}\]
where the minima are taken with respect to all possible representations of $h$ in the form (\ref{eq:4}) and (\ref{eq:5}), respectively.
\end{thm}
\begin{proof}[Proof of Theorem \ref{thm:4}]
The implication $(i) \Rightarrow (iii)$ has already been shown in the proof of Theorem \ref{thm:2}. In the same proof we have also seen that there exists a representation of $h \in \mathcal{K}$ in the form (\ref{eq:5}) such that $\|\phi'\|_{\hil'} \|\psi'\|_{\hil'} = \|Bh\|_{\mathcal{K}}$. The implication $(iii) \Rightarrow (ii)$ is trivial, as is the fact that
\[  \inf   \|\phi\|_\hil \|\psi\|_\hil \leq \inf \|\phi'\|_{\hil'}\|\psi'\|_{\hil'} .\]
Hence, what remains to be shown is the implication $(ii) \Rightarrow (i)$ and that for an arbitrary $h$ of the form (\ref{eq:4}) one has \[\|Bh\|_{\mathcal{K}} \leq \|\phi\|_{\hil} \|\psi\|_{\hil}.\]

To this end, suppose that $h(w)=1 + w\langle (I-wA)^{-1} \phi, \psi \rangle_{\hil}$ where $\phi,\psi \in \hil$ and $A$ is a contraction on $\hil$. By the Sz.-Nagy dilation theorem (see e.g. \cite{b_Davies}, Theorem 10.3.1) there exists a Hilbert space $(\hil',\langle . , . \rangle_{\hil'})$ with $\hil \subset \hil'$ (and $\langle f,g \rangle_{\hil'}=\langle f, g \rangle_{\hil}$ for every $f,g \in \hil$) and a unitary operator $U$ on $\hil'$ such that for every $n \in \mn_0$ and $f,g \in \hil$
\[  \langle U^n f, g \rangle_{\hil'} = \langle A^n f, g \rangle_{\hil}.\]
In particular, using the Neumann series we see that
\begin{eqnarray}
  h(w) &=&  1 + \sum_{n \geq 1} w^n \langle A^{n-1} \phi, \psi \rangle_{\hil}
=  1 + \sum_{n \geq 1} w^n \langle U^{n-1} \phi, \psi \rangle_{\hil'} \nonumber\\
&=& 1 + w\langle (I-wU)^{-1} \phi, \psi \rangle_{\hil'}. \label{eq:8}
\end{eqnarray}
Now let $E_U(.)$ denote the projection-valued spectral measure, defined on the Borel sets of $\mc$, associated to $U$ via the spectral theorem. Note that $E_U$ is supported on $\mt$. Let us define a complex Borel measure $\mu \in \mathcal{M}$ by setting  $\mu(.)= \langle E_U(.) \phi, \psi \rangle_{\hil'}$.
\begin{rem}\label{rem:3}
   It is easy to see that $\|\mu\| \leq \|\phi\|_{\hil'}\|\psi\|_{\hil'}= \|\phi\|_\hil \|\psi\|_\hil$.
\end{rem}
The spectral theorem and identity (\ref{eq:8}) imply that we can rewrite $h$ as
\[ h(w)= 1 + w\int_{{\mt}} \frac{\mu(d\zeta)}{1-w\zeta}.\]

Hence, setting $\overline{\mu}(\Omega):=\mu(\Omega^*)$ where $\Omega^*=\{ \lambda : \overline{\lambda}\in \Omega \}$, we see that $Bh=K \overline{\mu}$. This shows that $Bh \in \mathcal{K}$ and so  $h \in \mathcal{K}$ by Lemma \ref{lem:1}. Moreover, Remark \ref{rem:3} implies that
\[\|Bh\|_{\mathcal{K}} \leq \|\overline{\mu}\|= \|\mu\| \leq \|\phi\|_\hil \| \psi\|_\hil.\]
This concludes the proof of Theorem \ref{thm:4}.
\end{proof}

\begin{rem}
While in this paper we have concentrated on Cauchy transforms of measures on the circle, let us remark that we can obtain analogous results for Cauchy transforms of measures on the line (we thank the anonymous referee for making us aware of this fact). Namely, if
$$ h(\lambda)= \int_{\mr} \frac{\mu(ds)}{s-\lambda}, \qquad \im(\lambda) > 0,$$
denotes the Cauchy (or Borel) transform of a finite complex Borel measure $\mu$ on $\mr$ (which we assume to be normalized, $\mu(\mr)=1$, and to satisfy
$ \| s \mu(ds) \| := \int_{\mr} |s| d|\mu|(s) < \infty$), then we can show, using the same method as above, that
$$ \sum_{\lambda \in \mathcal{Z}(h)} \im(\lambda) \leq \|s \mu(ds)\|.$$
This follows from the fact that this class of functions coincides with the class of functions of the form
$ h(\lambda)= \langle (A-\lambda)^{-1} \phi, \psi \rangle,$ where $A$ is a selfadjoint operator on a Hilbert space $\hil$ and $\phi,\psi \in \hil$, and by an application of Theorem \ref{thm:3} to the operator $L=A+M$, where $Mf=-\langle f,\psi\rangle \phi$. The equality of these two classes of functions is a consequence of the spectral theorem for selfadjoint operators or, ultimately, of the Herglotz representation theorem.
\end{rem}

\section{Appendix}

\begin{proof}[Proof of Lemma \ref{lem:1}]
Let $h=K\mu$ for some $\mu \in \mathcal{M}$. Then
\begin{eqnarray*}
  (BK\mu)(w) = \frac 1 w \left( \int_{\mt} \frac{\mu(d\zeta)}{1-w\overline{\zeta}} - \int_{\mt} \mu(d\zeta) \right)
= \int_{\mt} \frac{\overline{\zeta}\mu(d\zeta) }{1-w\overline{\zeta}}\:,
\end{eqnarray*}
so $Bh =K(\overline{\zeta}\mu(d\zeta)) \in \mathcal{K}$. To show the converse let $Bh=K\sigma$ for some $\sigma \in \mathcal{M}$.
Define $\mu \in \mathcal{M}$ by
\[ \mu(d\zeta)=\zeta \sigma(d\zeta) + (h(0)-c_\sigma) m(d\zeta),\]
where $c_\sigma = \int_{\mt} \zeta \sigma(d\zeta)$.
Then
\begin{eqnarray*}
 \int_{\mt} \frac {\mu(d\zeta)} {1-w\overline{\zeta}}
&=& \int_{\mt} \frac {\zeta} {1-w\overline{\zeta}} \sigma(d\zeta) + (h(0)-c_\sigma) \int_{\mt} \frac{m(d\zeta)}{1-w\overline{\zeta}}  \\
&=& h(0)+ \int_{\mt} \frac {\zeta} {1-w\overline{\zeta}} \sigma(d\zeta) -c_\sigma \\
&=& h(0)+ \int_{\mt} \frac {\zeta} {1-w\overline{\zeta}} \sigma(d\zeta) -\int_{\mt} \zeta \frac{1-w\overline{\zeta}}{1-w\overline{\zeta}} \sigma(d\zeta) \\
&=& h(0)+  w \int_{\mt} \frac {\sigma(d\zeta)} {1-w\overline{\zeta}} =h(w).
\end{eqnarray*}
This shows that $h \in \mathcal{K}$.
\end{proof}

\begin{proof}[Proof of Theorem \ref{thm:3}]
The trace norm of $L-A$ can be computed as follows:
\begin{equation}\label{eq:tr}
  \|L-A\|_{tr} = \sup_{\{e_n\},\{f_n\}} \left\{ \sum_{n}| \langle (L-A) e_n, f_n \rangle| \right\},
\end{equation}
where the supremum is taken with respect to arbitrary orthonormal sequences $\{e_n\}$ and $\{f_n\}$ in $\hil$, see \cite{b_Simon05} Proposition 2.6. Let $ \lambda_1, \lambda_2, \ldots $  denote an enumeration of $\sigma_d(L)$, where each eigenvalue is counted according to its algebraic multiplicity and equal eigenvalues are denoted successively. By Schur's lemma, see e.g. \cite{b_Gohberg69} Remark 4.1 on page 17, there exists an orthonormal sequence $\{g_n\}$ in $\hil$ and a sequence of complex numbers $\{b_{jk}\}$ such that
\begin{equation}
\label{eq:2}
  Lg_n=b_{n1}g_1 + b_{n2}g_2 + \ldots + b_{nn}g_n \quad \mbox{and} \quad   b_{nn}=\lambda_n.
\end{equation}
Hence, we can use (\ref{eq:tr}) to obtain
\[
    \|L-A\|_{tr} \geq \sum_{n} |\langle (L-A) g_n,g_n \rangle| = \sum_{n} |\langle Lg_n,g_n \rangle - \langle Ag_n,g_n \rangle|.
\]
Using (\ref{eq:2})  we have $\langle Lg_n, g_n \rangle = \lambda_n$ and so
\begin{eqnarray*}
    \|L-A\|_{tr} \geq \sum_{n} |\lambda_n - \langle Ag_n,g_n \rangle| \geq \sum_{n} \dist(\lambda_n,\num(A)).
\end{eqnarray*}
\end{proof} 

\bibliography{bibliography}
\bibliographystyle{plain}

\end{document}